\documentclass[12pt]{amsart}
\usepackage{amsmath,amssymb,latexsym, amsthm, amscd, mathrsfs, stmaryrd}
\usepackage[all]{xy}

\theoremstyle{definition}
\theoremstyle{plain}
\newtheorem{prop}[subsection]{Proposition}
\newtheorem{thm}[subsection]{Theorem}
\newtheorem{lem}[subsection]{Lemma}

\newcommand{\Hom}{\mathrm{Hom}}
\newcommand{\mbf}{\mathbf}
\newcommand{\mbb}{\mathbb}
\newcommand{\mrm}{\mathrm}

\newcommand{\A}{\mathcal  A}
\newcommand{\q}{\sqrt q}
\newcommand{\C}{\mathscr C}
\newcommand{\I}{\mathcal I}

\newcommand{\End}{\mrm{End}}

\title[]{A geometric construction of generalized $q$-Schur algebras }

\author{Stephen Doty}
\address{Mathematics and Statistics\\ Loyola University Chicago\\ Chicago, IL 60626, USA}
\email{doty@math.luc.edu}

\author{  Yiqiang Li}
\address{Department of Mathematics\\   University  at Buffalo, SUNY\\
244 Mathematics Building, Buffalo, NY 14260}
\email{yiqiang@buffalo.edu}

\date{\today}

\begin{document}
\begin{abstract}
  We show that the algebras $\C(X\times X)$ in ~\cite{Li10} and
  $\mathscr L_d$ in ~\cite{Li12} are generalized $q$-Schur algebras as
  defined in ~\cite{D03}.  
  This provides a geometric construction of generalized $q$-Schur algebras in types $\mbf A,\mbf  D$ and $\mbf E$.
  We give a parameterization of Nakajima's
  Lagrangian quiver variety of type $\mbf D$ associated to a certain
  highest weight.
\end{abstract}

\maketitle

\section*{Introduction}
Be{\u\i}linson, Lusztig, and MacPherson \cite{BLM} gave a geometric
construction of the $q$-Schur algebra in type $\mbf A$ in terms of the
relative position of pairs of flags on a finite dimensional vector space
over a finite field of $q$ elements. (See also \cite{Du95}.) The question
naturally arises: can a similar construction be made in other types? We
will show that this question admits a positive answer in types $\mbf A, \mbf D$, and 
$\mbf E$.

Generalized Schur algebras for arbitrary type were introduced by
S.~Donkin in \cite{Don86}, and their $q$-analogues (the generalized
$q$-Schur algebras) were studied in \cite[29.2]{L93}, \cite{DS94}, and
\cite{D03}.

In ~\cite{Li10}, a finite dimensional quotient $\C$ of the quantum
algebra of type $\mbf D_{m+2}$ is constructed geometrically by using
functions on pairs of ramified partial flag varieties.  This algebra
enjoys many properties similar to that of generalized $q$-Schur
algebras constructed algebraically in ~\cite{D03}. We show that the
algebra $\C$ is a certain generalized $q$-Schur algebra of type $\mbf
D_{m+2}$; hence this generalized $q$-Schur algebra admits a geometric
construction.  We first use an argument similar to that in
~\cite[2.26, 2.27]{L03} to show that there is a surjective algebra
homomorphism from a generalized $q$-Schur algebra in ~\cite{D03} to
$\C$. Then we apply the fact that (the rational form of) generalized
$q$-Schur algebras are semisimple to obtain the injectivity result.

We also explain how a similar argument can be used to show that the
algebra $\mathscr L_d$ in ~\cite{Li12} is isomorphic to the integral
form of a generalized $q$-Schur algebra of finite type.  In the final
part of this note, we obtain a parametrization of Nakajima's
Lagrangian quiver variety of type $\mbf D_{m+2}$ by using the
connected components of the ramified partial flag variety studied in
this note.

\section{The algebras $\C$, $\dot{\C}$} 

\subsection{} \label{typeD}
Recall from ~\cite{Li10} that we have the following data.
\begin{itemize}

\item A Dynkin graph of type $\mbf D_{m+2}$:

\xymatrix{
& & i  \ar @{-}[d] \\
&      k   &  \ar@{-}[l]  j_1 \ar @{-}[r]  & j_2 \ar@{-}[r] &\cdots \ar@{-}[r]&j_m,
}
whose vertex set is denoted by $I$ and the associated Cartan matrix is $C=(c_{ab})_{a, b\in I}$.

\item A finite field $\mbb F_q$  of $q$ elements.

\item A fixed $d$-dimensional vector space $D$ over $\mbb F_q$.

\item A set $X$ of all `ramified' flags in $D$ of the form
\[
0\subseteq V_{j_m} \subseteq \cdots \subseteq V_{j_1} \subseteq U_i, U_k \subseteq U_{j_1} \subseteq \cdots \subseteq  U_{j_m} \subseteq D.
\]

\item A partition $X=\sqcup_{\nu\in \mbb N[I]} X_{\nu}$ of $X$, where $X_{\nu}$ contains all flags $U\in X$ subject to the conditions: 
\[
\dim U_i =\nu_i, \dim U_k =\nu_k, \dim U_{j_{\beta}} +\dim V_{j_{\beta}} =\nu_{j_{\beta}}, \quad \forall 1\leq \beta \leq m.
\]

\end{itemize}
Notice that $X_{\nu}$ is empty for all but finitely many $\nu$ in $\mbb N[I]$.  
As in ~\cite{Li10}, we consider the $\mbb C$-vector space $\C'(X\times X)$ of all $\mbb C$-valued functions on  $X\times X$. 
The vector space $\C'(X\times X)$ admits an associative algebra structure with the multiplication given by the convolution product:
\[
f_1 \circ f_2 (U, \tilde U) =\sum_{U'\in X} f_1(U, U') f_2 (U', \tilde U),
\quad \forall f_1, f_2\in \C'(X\times X); U, \tilde U\in X.
\]
The algebra $\C' (X\times  X)$ has a unit $\mbf 1$ defined by 
\begin{equation*}
 \mbf 1 (U, \tilde U) = 
\begin{cases} 
1,    & \mbox{if} \;   U = \tilde U,\\
0, & \mbox{otherwise}.
\end{cases}
\end{equation*}

For convenience, we shall write $|V|$ for the dimension of a given vector space $V$. We shall write ``$U\overset{a}{\subset} \tilde U$'' to denote that
$U$ is contained in $\tilde U$ and $U$ is one-dimension short of $\tilde U$ at the position $a\in I$.  The notation 
``$U \overset{j_1}{\subset} \tilde U, V_{j_1} \overset{1}{\subset} \tilde V_{j_1}$'' 
denotes  a pair $(U, \tilde U)$ in $X\times X$ satisfying 
that $U_a= \tilde U_a$ for any $a\in I$, $V_a =\tilde V_{a} $ for any $a\neq j_1$,  $V_{j_1} \subset \tilde V_{j_1}$ and 
$\dim V_{j_1} +1 = \dim \tilde V_{j_1}$. 

We are mainly interested in the subalgebra $\C \equiv \C(X\times X)$ generated by the following  functions $E_a$, $F_a$, and $K_{a}^{\pm 1}$ for any $a\in I$.  For $a= i, k$, the functions $E_a$ and $F_a$ are defined by
\begin{equation*}
\begin{split}
 E_a (U, \tilde U) = 
\begin{cases} 
\q^{-(| \tilde U_{j_1} | - | \tilde U_a | )},   & \mbox{if} \;   U \overset{a} {\subset} \tilde U,\\
0 & \mbox{otherwise}, 
\end{cases} \quad 
 F_a (U, \tilde U) =
\begin{cases}
\q^{-( |\tilde U_a | - | \tilde V_{j_1} | )},   & \mbox{if} \; U \overset{a} {\supset} \tilde U,\\
0 & \mbox{otherwise}.
\end{cases}
\end{split}
\end{equation*}
For $a= j_{\beta}$ and $1\leq \beta\leq m$, the functions $E_a$ and $F_a$ are defined as follows. 
\begin{equation}
\begin{split}
& E_{j_1} (U, \tilde U) =
\begin{cases}
\q^{- ( |\tilde U_i| + |\tilde U_k| - |\tilde V_{j_1} | - | \tilde V_{j_2}|)},   
& \mbox{if}\; U \overset{j_1}{ \subset}  \tilde U, V_{j_1} \overset{1}{\subset} \tilde V_{j_1},\\
\q^{- ( |\tilde U_{j_2} | +2 |\tilde V_{j_1} | - | \tilde U_{j_1}| ) },           
&  \mbox{if}\; U \overset{j_1}{\subset} \tilde U, U_{j_1} \overset{1}{\subset} \tilde U_{j_1},\\
0, & \mbox{otherwise.}
\end{cases}\\
&E_{j_{\beta}} (U, \tilde U) = 
\begin{cases}
\q^{-( |\tilde U_{j_{\beta-1}} | - |\tilde V_{j_{\beta}} |  - | \tilde V_{j_{\beta+1}} | ) },  
& \mbox{if}\;  U\overset{j_{\beta}}{\subset}  \tilde U,  V_{j_{\beta}} \overset{1}{\subset} \tilde V_{j_{\beta}},\\
\q^{ - ( | \tilde U_{j_{\beta+1}}| +2 |\tilde V_{j_{\beta}}| - | \tilde U_{j_{\beta}}| - | \tilde V_{j_{\beta-1}} | ) },  
& \mbox{if}\; U\overset{j_{\beta}}{\subset}  \tilde U,  U_{j_{\beta}} \overset{1}{\subset} \tilde U_{j_{\beta}},\\
0, & \mbox{otherwise.}
\end{cases}
\end{split}
\end{equation}

\begin{equation}
\begin{split}
& F_{j_1} (U, \tilde U) =
\begin{cases}
\q^{-( |\tilde U_{j_2}| - 2 | \tilde U_{j_1} | + |\tilde V_{j_1} |)}, 
& \mbox{if} \;  U \overset{j_1}{\supset} \tilde U, V_{j_1} \overset{1}{\supset} \tilde V_{j_1},\\
\q^{-( |\tilde U_{j_1}| - |\tilde U_i| - |\tilde U_k| -|\tilde V_{j_2} | )}, 
& \mbox{if} \;  U \overset{j_1}{\supset} \tilde U, U_{j_1} \overset{1}{\supset} \tilde U_{j_1},\\
0, & \mbox{otherwise.}
\end{cases}\\
&F_{j_{\beta}} (U, \tilde U) =
\begin{cases}
\q^{-(|\tilde U_{j_{\beta+1}}| -2 |\tilde U_{j_{\beta}}| + |\tilde V_{j_{\beta}}| +| \tilde V_{j_{\beta-1}} |)},
& \mbox{if}\;  U \overset{j_{\beta}} {\supset} \tilde U, V_{j_{\beta}} \overset{1}{\supset} \tilde V_{j_{\beta}},\\
\q^{-(-|\tilde V_{j_{\beta+1}}| + |\tilde U_{j_{\beta}} | - |\tilde U_{j_{\beta-1}}|)}, 
&\mbox{if}\;  U \overset{j_{\beta}} {\supset} \tilde U, U_{j_{\beta}} \overset{1}{\supset} \tilde U_{j_{\beta}},\\
0, & \mbox{otherwise.}
\end{cases}
\end{split}
\end{equation}
The functions $K_a^{\pm 1}$ are given by 
\begin{equation}
\label{K}
\begin{split}
&K_a^{\pm 1} (U, \tilde U)=
\begin{cases}
\q^{\pm ( |\tilde U_{j_1}| +|\tilde V_{j_1}| - 2 |\tilde U_a| ) },   & \mbox{if} \; U=\tilde U, a=i \, \mbox{or} \, k,\\
\q^{\pm (|\tilde U_{j_2}| + |\tilde V_{j_2}| + |\tilde U_i| +|\tilde U_k| -2 |\tilde U_{j_1}| -2|\tilde V_{j_1}| ) },   & \mbox{if} \; U=\tilde U, a=j_1,\\
\q^{\pm ( |\tilde U_{j_{\beta+1}}| +|\tilde V_{j_{\beta+1}}| + |\tilde U_{j_{\beta-1}}| +|\tilde V_{j_{\beta-1}}|- 2 |\tilde U_{j_{\beta}}| - 2|\tilde V_{j_{\beta}}| )},  
     & \mbox{if} \; U=\tilde U, a=j_{\beta},\\
0  & \mbox{if}\; U\neq \tilde U,
\end{cases}\\
\end{split}
\end{equation}
where $\beta$ runs from $2$ to $m$.

In addition to the above functions, we define 
\begin{equation*}
 \mbf 1_{\nu} (U, \tilde U) = 
\begin{cases} 
1,    & \mbox{if} \;   U = \tilde U \in X_{\nu}, X_{\nu} \; \mbox{nonempty},\\
0, & \mbox{otherwise},
\end{cases}
\end{equation*}
for any $\nu\in \mbb N[I]$. It is clear that $\mbf 1_{\nu} \mbf 1_{\nu'} =\delta_{\nu, \nu'} \mbf 1_{\nu}$ and  
\begin{equation}
\label{identity}
\mbf 1 =\sum_{\nu\in \mbb N[I]} \mbf 1_{\nu}.
\end{equation}
Let $\dot{\C}$ be the subalgebra of $\C'(X\times X)$ generated by 
the functions $\mbf 1_{\nu}$, $E_a  \mbf 1_{\nu}$ and $F_a \mbf 1_{\nu}$ for any $\nu\in \mbb N[I]$ and $a\in I$. 
By (\ref{identity}), we see that the algebra $\dot{\C}$ is unital.  This fact implies that  we have 
$\C \subseteq \dot{\C}$. 
Moreover, we have

\begin{lem}
$\C =\dot{\C}$. 
\end{lem}

\begin{proof}
We only need to show that $\mbf 1_{\nu}$ is in $\C$ for any $\nu\in \mbb N[I]$.  
This can be shown by an argument similar to the proof of Lemma 2.24 in ~\cite{L03}. 
For the sake of completeness, we shall provide the proof here.  Note that one can also prove this Lemma by an argument similar to the proof of Lemma 3.2 (i) in ~\cite{D03}.

Since the functions    $\mbf 1_{\nu}$ for any $\nu\in \mbb N[I]$  such that $\mbf 1_{\nu}\neq 0$ are orthogonal idempotents, we have 
\begin{equation}
\label{linear}
\prod_{a\in I} K_a^{n_a}  = \sum_{\nu\in \mbb N[I]} \q^{\sum_{a\in I} n_a b_{a, \nu}} \mbf 1_{\nu},
\end{equation}
for any $(n_a)\in \mbb Z^I$ where  $b_{a,\nu}$ are  the exponents of $\q$ in the definition of $K_a$. 
The sum  $\sum_{a\in I} n_a b_{a, \nu}$ can be rewritten as 
\[
\sum_{a\in I} d_a n_a + (n_a) \cdot C \nu, 
\]
where $C\nu$ is a vector in $\mbb Z^I$ whose  $a$-th component is equal to $\sum_{b\in I} c_{ab}\nu_b$ and the dot  is the standard inner product of two vectors. 
So the identity  (\ref{linear}) can be rewritten as 
\begin{equation}\label{linear-2}
\sum_{\nu\in \mbb N[I]}  \q^{ (n_a)\cdot C \nu} \mbf 1_{\nu}=
\q^{-\sum_{a\in I} d_a n_a } \prod_{a\in I} K_a^{n_a}.
\end{equation}
It is enough to show that we can find a vector $(n_a) \in\mbb Z[I]$ such that 
$(n_a) \cdot C \nu \neq (n_a)\cdot  C \nu'$ for any $\nu\neq \nu'$ such that $1_{\nu}$ and $1_{\nu'}$ are not zero.
This is because if such a vector $(n_a)$ exists, we can form together with (\ref{identity}) a linear system from (\ref{linear-2}) by considering the vectors
$(cn_a)$ for $c\in \mbb N$. It is clear that the  associated  coefficient matrix of the linear system is the Vandermonde matrix. Now by choosing the right number of the integers $c$,  we can get a square Vandermonde matrix which is invertible by our choice of the vector $(n_a)$.  This implies that 
$\mbf 1_{\nu}$ can be expressed as a linear combination of the functions $\q^{-\sum_{a\in I} d_a n_a } \prod_{a\in I} K_a^{n_a}$.

We return to the proof of the existence of such a vector $(n_a)$.  
Since $A$ is positive definite, the vector $C(\nu -\nu')$ is non zero for any $\nu\neq \nu'$. 
Since there are only finitely many  $\nu$ such that $\mbf 1_{\nu}$ is nonzero, we see that 
the collection of vectors $C(\nu-\nu')$  such that $\mbf 1_{\nu}$ and $\mbf 1_{\nu'} $ are non zero is finite. 
A standard argument  in linear algebra shows that we can find a vector
$(n_a) \in \mbb Z^{I}$ satisfying the requirement. The Lemma follows.
\end{proof}

\subsection{}

Let $\mbf U_{\q}$ be the specialization of the quantum algebra of type $\mbf D_{m+2}$ at $\q$. 
This is an associative algebra over $\mbb C$ generated by the symbols $E_a$, $F_a$ and $K_a^{\pm 1}$ for $a\in I$ and subject to the following 
defining relations.
\begin{align*}
&K_aK^{-1}_a=1,  K_aK_b=K_bK_a.\\
&K_aE_b=\q^{c_{ab}}E_bK_a,   K_aF_b=\q^{-c_{ab}}F_bK_a. \\
&E_aF_b-F_bE_a=\delta_{ab} \frac{K_a-K^{-1}_a}{\q-\q^{-1}}. \\
&E_a^2 E_b -(\q+\q^{-1}) E_a E_b E_a + E_b E_a^2=0;\\
  & F_a^2F_b -(\q+\q^{-1}) F_a F_bF_a + F_b F_a^2=0, & &\mbox{if}\; c_{ab}=-1.\\
&E_aE_b=E_bE_a,   F_aF_b=F_bF_a  &&\mbox{if} \quad c_{ab}=0.
\end{align*}

Let $L(\lambda)$ be the simple $\mbf U_{\q}$-module  of highest-weight $\lambda =\sum_{a\in I} \lambda_a a\in \mbb N[I]$. 
This is a $\mbf U_{\q}$-module generated by a vector $\xi_{\lambda}$ and  subject to the condition:
\[
K_a \xi_{\lambda} = \q^{\lambda_a} \xi_{\lambda}, \quad
E_a \xi_{\lambda} =0, \quad \forall a\in I. 
\]
We denote by $\I_D$ the two-sided ideal of $\mbf U_{\q}$ consisting of all elements $u$ in $\mbf U_{\q}$ such that $u. L(\lambda)=0$ for any $\lambda \in \mbb N[I]$ 
satisfying $\lambda = d j_{m} - C \nu$ for some $\nu\in \mbb N[I]$.  
Note that the quotient algebra $\mbf U_{\q}/\I_D$ is a generalized $q$-Schur algebra studied in ~\cite{D03} 
with the saturated set $\pi$ generated by the dominant weight $d j_m$ (or rather $d \omega_{j_m}$ in the notation of ~\cite{D03}). 
Indeed, it can be shown that 
\begin{equation}
\label{pi}
\pi =\{\lambda\in \mbb N[I] | \lambda =d j_{m} - C\nu, \quad \forall \nu\in \mbb N[I]\}. 
\end{equation}
 In the language of ~\cite{D03}, $\pi =\{ \lambda \in X^+| \lambda = d \omega_{j_m} - \sum_{a\in I} \nu_a \alpha_a, \nu_a\in \mbb N\}$. 
Note that $\pi$ is a finite set.
Indeed, a necessary condition for $\nu\in \mbb N[I]$ subject to  $dj_m - C\nu \in \mbb N[I]$  is that 
\begin{equation}\label{necessary}
\nu_i+ \nu_k \leq \nu_{j_1}\leq \cdots \leq \nu_{j_m} \leq d. 
\end{equation}
(A direct computation shows that $\pi$ is cofinal in the case when $m=2$, and is not in the case when $m>2$.)

Recall from ~\cite{Li10} that we have a surjective  algebra homomorphism 
\[
\Phi: \mbf U_{\q} \to \C,
\]
sending the generators in $\mbf U_{\q}$ to the respective elements in $\C$.   We have 

\begin{lem}
\label{factor}
The morphism $\Phi$ factors through a surjective algebra homomorphism 
\[
\Psi: \mbf U_{\q}/ \I_D \to \C.
\] 
\end{lem}

\begin{proof}
With respect to the partition $X=\sqcup X_{\nu}$, the algebra $\C$ admits a decomposition 
\[
\C =\oplus_{\nu, \tilde\nu} \C(X_{\nu} \times X_{\tilde \nu}), 
\]
where $\C(X_{\nu}\times X_{\tilde \nu})= \C \cap \C' (X_{\nu} \times X_{\tilde \nu})$. It is clear from the definitions that 
\[
K_a f = \q^{\delta_{a, j_m} d -  C\nu} f, \quad \forall a\in I, f\in  \C(X_{\nu}\times X_{\tilde \nu}). 
\]
This implies that $\I_D \subseteq \ker (\Phi)$.  The Lemma follows.
\end{proof}

\begin{thm}\label{main}
The algebra homomorphism $\Psi: \mbf U_{\q}/ \I_D \to \C$ is an isomorphism. 
\end{thm}

\begin{proof}
By Corollary 3.13 in ~\cite{D03}, the algebra $\mbf U_{\q}/\I_D$ has a  presentation by generators and relations.  
The generators are $E_a$, $F_a$ for any $a\in I$  and $1_{\mu}$ for $\mu \in W\pi$ where $W$ is the Weyl group of type $\mbf D_{m+2}$. 
If $\mu=\sum_{a\in I} \mu_a a$ (or $\sum_{a\in I} \mu_a \omega_a$ in ~\cite{D03}), we see that 
\[
\Psi (1_{\mu}) =\mbf  1_{\nu}, \quad \mbox{where} \; \mu = d j_m - C\nu. 
\] 
This is guaranteed by comparing  the defining relations of $\mbf U_{\q}/\I_D$ in ~\cite[1.3]{D03}  with the definition of $\mbf 1_{\nu}$. 
By Propositions 3.8 and 3.10 in ~\cite{D03}, we see that the algebra $\mbf U_{\q}/\I_D$ is a finite dimensional  semisimple algebra and
\begin{align}
\label{semisimple}
\mbf U_{\q}/\I_D \simeq \bigoplus_{\lambda\in \pi}  \End (L(\lambda)),
\end{align}
where $\pi$ is defined in (\ref{pi}). Since $\mbf U_{\q}/\I_D$ is semisimple, so is $\C$.  This implies that $\C$ has a decomposition similar to
(\ref{semisimple}) where the sum runs over a subset of $\pi$.  Moreover, the homomorphism $\Psi$ is compatible with such decompositions. 
In order to show that $\Psi$ is an isomorphism, we only need to show that $\Psi(1_{\lambda})= \mbf 1_{\nu}$ is non zero for any $\lambda =d j_m - C\nu \in \pi$.
It is reduced to show that the variety $X_{\nu} $ is non empty for any $\lambda=dj_m - C\nu  \in \pi$. 
Note that a necessary condition for $\nu\in \mbb N[I]$ to be in $\pi$ is (\ref{necessary}). 
For any $\nu \in \mbb N[I]$ subject to the condition (\ref{necessary}), the associated variety $X_{\nu}$ is always nonempty from the definition. The proof is complete. 
\end{proof}

\section{The algebra $\mathscr L_d$}

\subsection{} 

We shall show that a similar argument proves that 
when the quiver is symmetric of finite type, i.e., a simply-laced Dynkin
diagram, the algebra $\mathscr L_d$ in ~\cite[6.5]{Li12} is also a generalized $q$-Schur algebra. 
In this situation, the letter $d$ stands for an element in $\mbb N[I]$ where $I$ is the vertex set of the fixed quiver.
Let $C$ denote the Cartan matrix of the underlying graph  of the quiver. 
If we set 
\begin{equation}\label{pi-2}
\pi =\{ \mu\in \mbb N[I] | \mu = d - C  \nu, \quad \nu\in \mbb N[I]\}. 
\end{equation}
Then  a similar proof as that of Lemma ~\ref{factor} shows that the algebra homomorphism $\Psi_d$ in ~\cite[6.6]{Li12} factors through 
the integral form $_{\A} \mbf S(\pi)$ in ~\cite[8.1, 8.2]{D03} of the generalized $q$-Schur algebra $\mbf S(\pi)$ determined by  the Cartan matrix $C$ and the saturated set $\pi$. Moreover the induced algebra homomorphism $_{\A} \mbf S(\pi) \to \mathscr L_d$ is surjective and sends
the generators $1_{\mu} $ to the isomorphism class of the complex of sheaves $\mathscr I_ {\mu'}$ in ~\cite[5.1]{Li12} where $\mu' = d- \nu$ if 
$\mu = d- C\nu$.
 
 Just like the proof of Theorem ~\ref{main}, we only need to show that $\mathscr I_{\mu'}$ is non zero for any $\mu\in \pi$. 
Now the condition  $\mu = d-C\nu\in \pi$ is equivalent to the condition
\[
d_i +\sum_{j}  \nu_j -2\nu_i \geq 0,\quad \forall i\in I. 
\]  
where the sum runs over all $j\neq i$ such that $c_{ji} =-1$.  This immediately implies that  $d_i +\sum_{j}  \nu_j \geq \nu_i$ for any $i\in I$.
The latter condition guarantees that there is a nonzero element $P$ in $\mathscr D_{\mbf G}( \mbf E_{\Omega}(D, V))$ in ~\cite[6.8]{Li12} such that
the $I$-graded dimension of $V$ is $\nu$.  By the definition of $\mathscr I_{\mu'}$, we have 
$\mathscr I_{\mu'} P = P \neq 0$.  This property rules out the
possibility of $\mathscr I_{\mu'}=0$.  In summary, we
have proved

\begin{thm}
The algebra $\mathscr L_d$ in ~\cite[6.5]{Li12} is the integral form $_{\A} \mbf S(\pi)$ of a  generalized $q$-Schur algebra  in ~\cite[8.1, 8.2]{D03} 
(see also ~\cite{DS94}) where $\pi$ is defined in (\ref{pi-2}).
\end{thm}

\noindent
{\it Remark.}
If we  choose the quiver such that the associated Cartan matrix $C$ is of type $\mbf D_{m+2}$, i.e., the same as that  of Section ~\ref{typeD},
and the element $d$ is taken to be $\dim D  j_m$, then  the complexified algebra $\mbb C\otimes \mathscr  L_d$ ($v$ is specialized to $\q$) is isomorphic to $\C$. This is because both algebras are isomorphic to the same generalized $q$-Schur algebra of type $\mbf D_{m+2}$.  It will be very interesting to make a direct connection of the two algebras $\C$ and $\mathscr L_d$ in ~\cite{Li10} and ~\cite{Li12}, respectively.  
Note that in type $\mbf A_n$ case, the algebra $\mathscr L_d$  for certain $d$ is shown in ~\cite[Section 8]{Li12} to be isomorphic to the $q$-Schur algebra.

\section{A parametrization}

\subsection{}
If the ground field $\mbb F_q $ is replaced by its algebraic closure $\mbb F$, the set $X$ becomes an algebraic variety over $\mbb F$. 
 We shall fix a mistake  in  ~\cite{Li10}. 
The dimension of the connected component  $X_{\nu, \underline c}$ of $X$  in ~\cite[4.1]{Li10}  is 
\[
\sum_{a=i, k} \nu_a (\nu_{j_1} -\nu_a) + \sum_{\beta=1}^m \nu_{j_{\beta}} (\nu_{j_{\beta+1}} -\nu_{j_{\beta}})-
\sum_{\beta=1}^m (c_{\beta}-c_{\beta+1} ) (\nu_{j_{\beta+1}} - \nu_{j_{\beta}} -c_{\beta+1} + c_{\beta}).
\]

\subsection{}

Let $\Psi_2: Y\to X$ be the $set$-$theoretic$  map defined in ~\cite[4.2]{Li10}.  We set
\[
Y_{\nu, \underline c} =\Psi_2^{-1} (X_{\nu, \underline c}).
\]
We have a partition of $Y$ into locally closed subsets:
\[
Y =\sqcup_{\nu, \underline c} Y_{\nu, \underline c}.
\]

\begin{lem}
\label{vectorbundle}
The restriction $\Psi_{2,\underline c}: Y_{\nu, \underline c} \to X_{\nu, \underline c}$ of $\Psi_2 $ to $Y_{\nu,\underline c}$  is a vector bundle of fiber dimension 
\[
\sum_{\beta=1}^m (c_{\beta}-c_{\beta+1} ) (\nu_{j_{\beta+1}} - \nu_{j_{\beta}} -c_{\beta+1} + c_{\beta}).
\]
\end{lem}

\begin{proof}
We will use the following fact. Fix a decomposition $E=E_1\oplus E_2$ of a vector space. Let $F_1$ and $F_2$ be a subspace of $E_1 $ and $E_2$, respectively.  Let $\mathcal F$ be the collection of all subspaces $F$ in $E$ such that $F\cap E_1 =F_1$ and $\pi_2(F)=F_2$ where $\pi_2: E\to E_2$ is the natural projection. Then $\mathcal F$ is isomorphic to the vector space $\Hom (F_2, E_1/F_1)$. A bijection $\phi \mapsto F(\phi)$ of the two spaces is defined by 
\[
F(\phi) =\{ f_1+ f_2+\phi(f_2) | f_1\in F_1, f_2\in F_2\},
\]
where we fix a decomposition $E= F_1\oplus E_1/F_1\oplus F_2 \oplus E_2/F_2$.

By using this fact, we see that the fiber $\Psi_2^{-1}( U) $ is the same as the collection of linear maps 
$(\phi_{\beta})_{1\leq  \beta \leq  m}$ in $\oplus_{\beta=1}^{m} \Hom (V_{j_{\beta}}, D/ U_{j_{\beta}})$ such that 
\begin{align}
&\bar \sigma ( \mathcal V (\phi_{\beta}) ) \subseteq U_{j_{\beta+1}};  \tag{a}\\
&V_{j_{\beta+1}} \subseteq \mathcal V(\phi_{\beta}) \cap 0\oplus D, \quad \forall 1\leq \beta \leq m. \tag{b}
\end{align}
The condition (a) holds if and only if $\phi_{\beta}(v) \in U_{j_{\beta+1}}$ for  $1\leq \beta\leq m$. 
The condition (b) holds if and only if $\phi_{\beta}(v) =0$ for any $v\in V_{j_{\beta+1}}$. Therefore, the fiber $\Psi_2^{-1}(U)$ is isomorphic to the vector space 
\[
\oplus_{\beta=1}^m \Hom (V_{j_{\beta}}/V_{j_{\beta+1}}, U_{j_{\beta+1}}/ U_{j_{\beta}}).
\]
The Lemma follows.
\end{proof}

A consequence of Lemma ~\ref{vectorbundle} is that $Y_{\nu, \underline c}$ is connected and smooth.
From this Lemma, we have 
\[
\dim Y_{\nu, \underline c} =\dim X_{\nu, \underline c} + \sum_{\beta=1}^m (c_{\beta}-c_{\beta+1} ) (\nu_{j_{\beta+1}} - \nu_{j_{\beta}} -c_{\beta+1} + c_{\beta})=\dim Y_{\nu}.
\]
So we have 

\begin{prop}
The irreducible components of $Y_{\nu}$ are the closure   $ \overline{Y_{\nu, \underline c}}$  of $Y_{\nu, \underline c}$ for any sequence 
$ \underline c=(c_m, \cdots, c_1)$ of  non-decreasing and non-negative integers such that $ c_1\leq \min\{\nu_i, \nu_k\} $. 
\end{prop}

\section{Acknowledgements}

Y.~Li thanks Professor Leonard Scott for asking whether the algebra $\C$ is
a generalized $q$-Schur algebra during his visit to University of Virginia
in 2010.  S.~Doty is partially supported by the Simons Foundation:
Collaboration Grant 245975 and Y.~Li is partially supported by the NSF
grant: DMS 1160351.


\begin{thebibliography}{99999}\frenchspacing

\bibitem[BLM]{BLM} A.A. Be{\u\i}linson, G. Lusztig, R. MacPherson,
  {\em A
  geometric setting for the quantum deformation of $\mathrm{GL}_n$}.
  Duke Math. J. {\bf 61} (1990), no. 2, 655--677.

\bibitem[Don86]{Don86} S. Donkin, 
                           {\em On Schur algebras and related algebras, I}.  
                          J. Algebra, {\bf 104} (1986), no.  2, 310--328.

\bibitem[D03]{D03} S.  Doty. 
            {\em  Presenting generalized q-Schur algebras.} Represent. Theory {\bf 7}  (2003), 196--213. 

\bibitem[DS94]{DS94} J. Du, L. Scott, 
           {\em Lusztig's conjectures, old and new. I.}  J. reine
           angew. Math. {\bf 455} (1994), 141--182.

\bibitem[Du95]{Du95} J. Du, 
{\em A note on quantized Weyl reciprocity at roots of unity.}
Algebra Colloq. {\bf 2} (1995), no. 4, 363--372.
      

\bibitem[Li10]{Li10} Y. Li, 
               {\em A geometric realization of quantum groups of type D}.  Adv. Math. {\bf 224} (2010), no. 3, 1071--1096. 
               
               
\bibitem[Li12]{Li12} Y. Li, {\em A geometric realization
                 of modified quantum algebras}. Preprint 2012.
               arXiv:1007.5384.
                              
\bibitem[L93]{L93} G. Lusztig,  {\em Introduction to
               quantum groups.} Progress in Mathematics,
               110. Birkh\"{a}user Boston, Inc., Boston, MA, 1993.

 \bibitem[L03]{L03} G. Lusztig, 
       {\em  Constructible functions on varieties attached to quivers.}
        Studies in memory of Issai Schur (Chevaleret/Rehovot, 2000), 177--223, Progr. Math., 210, Birkh\"{a}user Boston, Boston, MA, 2003.
                      

\end{thebibliography}
\end{document}